\documentclass[12pt,reqno]{amsart}
\usepackage{amsmath,amssymb,enumerate}
\usepackage[dvipdfmx]{graphicx}
\usepackage{color}
\usepackage{comment}
\usepackage[all]{xy}
\usepackage{bm}

\usepackage{url}

\newtheorem{tm}{Theorem}[section]
\newtheorem{lemma}[tm]{Lemma}

\newtheorem{remark}[tm]{Remark}

\newtheorem{proposition}[tm]{Proposition}

\makeatletter
\newcommand{\subscripts}[3]{%
  \@mathmeasure\z@\displaystyle{#2}%
  \global\setbox\@ne\vbox to\ht\z@{}\dp\@ne\dp\z@
  \setbox\tw@\box\@ne
  \@mathmeasure4\displaystyle{\copy\tw@_{#1}}%
  \@mathmeasure6\displaystyle{{#2}_{#3}}%
  \dimen@-\wd6 \advance\dimen@\wd4 \advance\dimen@\wd\z@
  \hbox to\dimen@{}\mathop{\kern-\dimen@\box4\box6}%
}
\makeatother
\makeatletter
\@namedef{subjclassname@2020}{\textup{2020} Mathematics Subject Classification}
\makeatother

\allowdisplaybreaks[3]

\makeatletter
 
 \@addtoreset{equation}{section}
\makeatother

\begin{document}
%%%%%%%%%%%%%%%%%%%%%%%%%%%%%%%%%%%%%%%%%%%
%\setlength{\baselineskip}{15.5pt}
%%%%%%%%%%%%%%%%%%%%%%%%%%%%%%%%%%%%%%%%%%%
\title[Volatility estimation from a view point of entropy]{
Volatility estimation from a view point of entropy
}
%%%%%%%%%%%%%%%%%%%%%%%%%%%%
\author[J. Akahori]{Jir\^o Akahori}
\author[R. Namba]{Ryuya Namba}
\author[A. Watanabe]{Atsuhito Watanabe}

\date{\today}
\address[J. Akahori]{Ritsumeikan University, 1-1-1 nojihagashi, Kusatsu, 525-8577, Japan}
\email{{\tt{akahori@se.ritsumei.ac.jp}}}
\address[R. Namba]{
Kyoto Sangyo University, Motoyama, Kamigamo, Kita-ku, Kyoto, 
603-8555, Japan}
\email{{\tt{rnamba@cc.kyoto-su.ac.jp}}}
\address[A. Watanabe]{
Ritsumeikan University, 1-1-1 nojihagashi, Kusatsu, 525-8577, Japan}
\email{{\tt{atsu.watanabe0507@gmail.com
}}}
\subjclass[2020]{Primary 62G05; Secondary 62G20.}

\keywords{volatility estimation, microstructure noise, SIML method, Malliavin--Mancino's Fourier estimator, likelihood function.}
%%%%%%%%%%%%%%%%%%%%%%%%%%%%

%
%
%%%%%%%%%%%%%%%%%%%%%%%%%%%%%%%%%%%%%%%%%%%%%%%%%%%%%%
\begin{abstract}
In the present paper, we first revisit the volatility estimation approach proposed by N. Kunitomo and S. Sato, and second, we show that the volatility estimator proposed by P.~Malliavin and M.E.~Mancino 
can be understood in a unified way by the approach. 
Third, we introduce an alternative estimator that might overcome the inconsistency caused by the microstructure noise of the initial observation. 
\end{abstract}
%%%%%%%%%%%%%%%%%%%%%%%%%%%%%%%%%%%%%%%%%%%%%%%%%%%%%%

\maketitle

\section{Introduction}

\subsection{High-frequency statistics under micro-structure noise}
Statistics of continuous-time stochastic processes
with stochastic calculus, as seen in
the textbooks by Prof. Albert N. Shiryaev, 
with Prof. Robert S. Liptser
\cite{LSI, LSII}, and with Prof. Jean Jacod \cite{JS}, 
has been the basis of the more recent topics
including so-called high-frequency statistics (HFS, for short).
The HFS is also based on,
instead of stationarity, the semi-martingale principle
\begin{align}\label{qv}
    \sum_k (X_{t_k} -X_{t_{k-1}})^2 \to  [ X ]_t \,
    \text{almost surely as $ \max_k |t_k -t_{k-1}| \to 0 $, }
\end{align}
for the data of the semi-martingale  $ X $ 
sampled at 
$ 0 = t_0 < t_1 < \cdots < t_n =t $.

The HFS has been intensively studied in view of applications in finance
because convergence \eqref{qv} has some kind of reality
in financial markets that allows high-frequency trading.
The reality, however, leaves us the problem of so-called {\em microstructure noise};
the convergence \eqref{qv} is not really achieved
but only with some modification, which can be well explained by supposing that we
observe with noise whose distribution is rather irrelevant to the frequency (see, e.g. \cite{AJ}).  

There has been a great deal of literature on the construction of an estimator of $ [X]_t $ under the microstructure noise. 
In the present paper, we
introduce an approach that relies on a likelihood function
along the line of the one proposed in \cite{KS08a}
(which we call the Kunitomo-Sato type likelihood function),
which is somewhat different from those found in
Section 7.3 of \cite{AJ}. 
The approach actually unifies
the volatility estimator 
proposed by N.~Kunitomo and S.~Sato \cite{KS08a}
and the one proposed by
P.~Malliavin and M.~E.~Mancino 
\cite{MMFS, MMAS}. 
Further, we will point out that
both fail to be a consistent estimator if there is also a micro-structure noise in the initial observation. 
We then propose an alternative estimator which overcomes the difficulty, based on the likelihood function approach. 

The rest of the present paper is organized as follows. 
In section \ref{sec:1}, 
we study the Kunitomo--Sato approach, which is referred to as the 
separating information maximum likelihood
method (SIML method, for short).
\if1
After giving the general setting and the problem of HFS in Section \ref{Problem}, 
we introduce the Kunitomo-Sato estimator in section \ref{SIMLintro}, together with 
the results obtained in \cite{KS08a}.
In section \ref{Inconsis1}, 
we point out that the assumption in \cite{KS08a} that the initial observation is noise-free is critical. 
In section \ref{KSL}, 
we review the {\em separation of information} argument in \cite{KS08a}, which is the above-mentioned likelihood function approach.  
\fi
In Section \ref{MM-KS}, 
we discuss how the Malliavin--Mancino estimator can be understood as a {\em separating information} estimator. 
In Section \ref{INA}, we introduce 
an alternative estimator which could be consistent under the initial noise. 
Section \ref{conc} concludes the paper.

\section{The SIML method}\label{sec:1}

\subsection{The Problem}\label{Problem}
%Throughout the present paper, we consider a complete probability space $ (\Omega, \mathcal{F}, \mathbf{P}) $, which is the canonical space of a $d$-dimensional Wiener process $ \mathbf{W} \equiv (W^1, W^2,  \cdots, W^d) $ on the time interval $[0,1]$. We denote by $ \mathcal{F}_t,\, t \in [0, 1], $ the complete $ \sigma $-algebra generated by $ \{\mathbf{W}_s : 0 \leq s \leq t \} $ and by $ L^2_a [0,1] $ the space of $ \{\mathcal{F}_t \} $-adapted processes $ \theta $ with $ \mathbf{E} [ \int_0^1 |\theta(s)|^2 \mathrm{d}s ] < +\infty $. 

In the present section, 
we describe our setting and the problem, which is typical in HFS. 
Let $J \in \mathbb{N}$. Consider an It\^o process 
\begin{equation}\label{Ito}
\begin{split}
    X^j_t = X^j_0 
    + \int_0^t b^j (s) \,\mathrm{d}s 
    + \sum_{r=1}^d \int_0^t \sigma^j_r (s) \, \mathrm{d} W^r_s,
\end{split}
\end{equation}
for $ j=1, 2, \dots, J $ and $ t \in [0,1] $, 
where $ b^j $ and $ \sigma^j_r $
are integrable and square integrable adapted processes, respectively, 
for all $ j =1, 2, \dots, J$ 
and $ r =1, 2, \dots, d $,
$ \mathbf{W} \equiv (W^1, W^2,  \cdots, W^d) $ is a Wiener process, 
all of which are defined on a 
filtered probability space $ (\Omega, \mathcal{F}, \mathbf{P}, \{ \mathcal{F}_t\} ) $. 

We suppose that 
the observations $Y$ are with some additive noise $v$, 
often called the {\it micro-structure noise}. Namely, 
\begin{equation}
    Y^j_{t^j_k} \equiv X^j_{t^j_k}+ v^j_k \label{microstructure noise}
\end{equation}
for $ k = 0, 1, \dots, n^j $ and $ j = 1, 2, \dots, J$, 
where $ \{v^j_k\}_{j, k} $
is a family of %i.i.d. 
random variables whose distribution 
will be specified later, and 
$ n^j $  is an integer. 

We are interested in  
constructing an estimator 
$ (V^{j,j', h})_{j, j'=1, 2, \dots, J} $
of $ h $-
integrated volatility matrix defined by
\begin{equation*}
    \int_0^1 h(s) \Sigma^{j,j'} (s) \,\mathrm{d}s
    :=\sum_{r=1}^d \int_0^1 h(s) \sigma_r^j (s) \sigma_r^{j'} (s)\, \mathrm{d}s,
\end{equation*}
for $ j, j'=1, 2, \dots, J $, 
where $ h $ is a given function
or 
out of the observations, 
which is consistent 
in the sense that each
$ V^{j,j'} $ converges to 
$ \int_0^t h(s) \Sigma^{j,j'} (s) \,{\rm d}s $
in probability as $ n \to \infty $, under the condition that 
\begin{equation}\label{rho}
    \rho_n := \max_{j,k} |t^j_k - t^j_{k-1}| \to 0
\end{equation}
as $ n \to \infty $.

\subsection{The SIML estimator}\label{SIMLintro}
Let us put
\begin{equation*}
    p^n_{k,l} = \sqrt{\frac{2}{n+ \frac{1}{2}}}
    \cos \left( \left(l - \frac{1}{2}\right)
    \pi \left( \frac{k-\frac{1}{2}}{n+\frac{1}{2}}
    \right) \right)
\end{equation*}
for $ k,l=1, 2, \dots, n $, $ n \in \mathbf{N} $. 
Note that $ \mathbf{P}_n = (p^n_{k,l}) $
forms an orthogonal matrix. 
The {\it separating information maximum likelihood} (SIML) estimator, introduced 
by Prof. Naoto Kunitomo and Prof. Seisho Sato 
in \cite{KS08a} (see also \cite[Chapters 3 \& 5]{KSK}), 
is given by 
\begin{equation}\label{SIMLestimator}
    \begin{split}
        V_{n,m_n}^{j,j'}
        := \frac{n}{m_n} \sum_{l=1}^{m_n} \left( \sum_{k=1}^{n^j}
        p^{n^j}_{k,l} \Delta Y^j_k \right)
        \left(\sum_{k'=1}^{n^{j'}} p^{n^{j'}}_{k',l} \Delta Y^{j'}_{k'} \right),  \qquad j, j'=1, 2, \dots, J,
    \end{split}
\end{equation}
where $ m_n (\ll n) $ is an integer,
and we understand 
$ \Delta $ to be the difference operator given by
$ (\Delta a)_k = a_k -a_{k-1} $ for a sequence $\{a_k\}_k$. 
We then write 
\begin{equation*}
\Delta Y^j_k = Y^j_{t^j_{k}} - Y^j_{t^j_{k-1}}, \qquad j, j'=1, 2, \dots, J, \, k=1, 2, \dots, n^j.
\end{equation*}
They have proved the following two properties:
\begin{enumerate}[{\bf (a)}]
    \item ({\it the consistency}):
    the convergence in probability of $ V^{j,j'}_{n, m_n} $ to $ \int_0^1 \Sigma^{j,j'} (s) \,\mathrm{d}s $ as $ n \to \infty $
    is attained, provided 
    that $ m_n = o (n^{1/2})$, and
    \item ({\it the asymptotic normality of the error}): the stable convergence of 
    \begin{equation*}
    \begin{split}
       & \sqrt{m_n}\left(V^{j,j'}_{n, m_n}-\int_0^1 \Sigma^{j,j'} (s) \,\mathrm{d}s \right) \\
        & \to N \left(0, \int_0^1 \left(\Sigma^{j,j}(s) \Sigma^{j',j'}(s) + (\Sigma^{j,j'} (s))^2 \right) \mathrm{d}s
        \right)
        \end{split}
    \end{equation*}
    holds true as $ n \to \infty$ 
    if $ m_n = o (n^{2/5})$. 
\end{enumerate}
when the following hold (see \cite{KSK} for details).
\begin{enumerate}[(i)]
    \item the observations 
are equally spaced, that is, $ t^j_k \equiv k/n $, 
\item $ \{v^j_k\}_{k=1}^n $ are
independently and identically distributed, and the common distribution is independent of the choices of $ n $ and $ j $,
\item $ v_0^j \equiv 0 $, 
\item $ \sigma $ is deterministic and $ \mu \equiv 0 $. 
\end{enumerate}

\begin{remark}
In \cite{ANW}, the consistency and the asymptotic normality are proven when $ b $ and $ \Sigma $ are adapted 
processes with some regularity under a general sampling scheme, but without micro-structure noise. 
\end{remark}

\begin{remark}
    In the series of papers 
    \cite{KS08b, KS10, KS2, KS13, KSK, MK15, KK21, KS22}, 
    N. Kunitomo and his collaborators intensively studied the practical applications as well as some extensions of the SIML method.
\end{remark}

\subsection{Inconsistency caused by the initial noise}\label{Inconsis1}
When $ v^j_0 \ne 0 $,
we emphasize that the consistency of the SIML estimator may fail. 
Let us see this
when $ J=1 $. 
We start with the decomposition of $  V_{n,m_n} \equiv V_{n,m_n}^{1,1} $ as
\begin{equation*}
    \begin{split}
        V_{n,m_n} \equiv V_{n,m_n}^{1,1} 
       & = \frac{n}{m_n} \sum_{l=1}^{m_n} \left( \sum_{k=1}^{n}
        p^{n}_{k,l} \Delta X_k \right)^2
        + \frac{2n}{m_n} \sum_{l=1}^{m_n} \left( \sum_{k=1}^{n}
        p^{n}_{k,l} \Delta X_k \right) \left( \sum_{k=1}^{n}
        p^{n}_{k,l} \Delta v_k \right) \\
      & + \frac{n}{m_n} \sum_{l=1}^{m_n} \left( \sum_{k=1}^{n}
        p^{n}_{k,l} \Delta v_k \right)^2. 
    \end{split}
\end{equation*}
Under the general sampling scheme for which consistency is established in \cite{ANW}, 
we know that 
\begin{align*}
    \begin{aligned}
        \int_0^1 \Sigma (s) \, \mathrm{d}s - \frac{n}{m_n} \sum_{l=1}^{m_n} \left( \sum_{k=1}^{n}
        p^{n}_{k,l} \Delta X_k \right)^2
    \end{aligned}
\end{align*}
converges to zero in $ L^2 $, and so its mean also converges to zero.
On the other hand, 
we have clearly
\begin{align*}
    \mathbf{E} \left[\frac{2n}{m_n} \sum_{l=1}^{m_n} \left( \sum_{k=1}^{n}
        p^{n}_{k,l} \Delta X_k \right) \left( \sum_{k=1}^{n}
        p^{n}_{k,l} \Delta v_k \right)\right] = 0.
\end{align*}
The following observation claims that we cannot establish the consistency
of the estimator with $ m_n = o(1) $ as $ n \to \infty $. 
\begin{proposition}\label{prop1}
For any pair of $ (n, m_n) $ such that $ m_n < (n+1)/2  $, we have 
\begin{align}\label{lb1}
    \begin{aligned}
         \mathbf{E} \left[  \frac{n}{m_n} \sum_{l=1}^{m_n} \left( \sum_{k=1}^{n}
        p^{n}_{k,l} \Delta v_k \right)^2\right] \geq \frac{1}{2} \mathbf{E}[(v_0)^2]. 
    \end{aligned}
\end{align}
\end{proposition}
\begin{proof}
The claim is almost obvious since it holds that
\begin{align*}
    \begin{aligned}
       & \mathbf{E}\left[  \left( \sum_{k=1}^{n}
        p^{n}_{k,l} \Delta v_k \right)^2\right] \\
        & =\mathbf{E}\left[  \left( \sum_{k=1}^{n-1}
        (p^{n}_{k,l} -p^n_{k+1,l}) v_k 
        + p^n_{n,l} v_n - p^n_{1,l} v_0 \right)^2\right] 
    %   & = \mathbf{E}[v_1^2]       \left( \sum_{k=1}^{n-1}        (p^{n}_{k,l} -p^n_{k+1,l})^2         + (p^n_{n,l})^2 + (p^n_{1,l})^2 \right)
       \geq \mathbf{E}[v_0^2] (p^n_{1,l})^2,
    \end{aligned}
\end{align*}
and 
\begin{align*}
    \begin{aligned}
   \frac{n}{m_n}
   \sum_{l=1}^{m_n}(p^n_{1,l})^2
   =   
   \frac{2n}{(n+ \frac{1}{2})m_n} \sum_{l=1}^{m_n} 
    \cos^2 \left(
     \frac{2l-1}{2n+1}\frac{\pi}{2} 
    \right). 
    \end{aligned}
\end{align*}
In fact, we have 
\begin{align*}
    \begin{aligned}
       & \sum_{l=1}^{m_n} 
    \cos^2 \left(
     \frac{2l-1}{2n+1}\frac{\pi}{2} 
     \right) = \frac{m_n}{2} 
    + \frac{1}{2} \sum_{l=1}^{m_n} \cos \left(\frac{(2l-1)\pi}{2n+1}\right)
    \end{aligned}
\end{align*}
and 
\begin{align*}
    \begin{aligned}
    \frac{1}{2} \sum_{l=1}^{m_n} \cos \pi\left(\frac{2l-1}
{2n+1}\right) 
&= \frac{1}{4} \sum_{l=1}^{m_n} \left( e^{i \pi \frac{2l-1}{2n+1}} + e^{-i \pi \frac{2l-1}{2n+1}} \right) \\
&= \frac{1}{4} \sum_{l=-m_n+1}^{m_n} e^{i \pi \frac{2l-1}{2n+1}} 
= \frac{e^{-i\pi \frac{2m_n -1}{2n+1}}}{4} 
\frac{e^{i \pi \frac{2m_n}{2n+1}}- e^{- i \pi \frac{2m_n}{2n+1}}}{e^{i \pi \frac{1}{2n+1}}-e^{-i \pi \frac{1}{2n+1}}} \\
&= \frac{1}{4} \frac{\sin \left( \frac{2m_n \pi}{2n+1} \right) }{\sin  \left(\frac{\pi}{2n+1}\right)}, 
%> \frac{2n+1}{\pi} \frac{m_n}{2n+1} = \frac{m}{\pi}
\end{aligned}
\end{align*}
which ensure \eqref{lb1}. 
\end{proof}

\subsection{The Kunitomo--Sato likelihood function}\label{KSL}

Let us introduce a prototype of Kunitomo--Sato (KS, for short) likelihood function.  
To focus on the study of estimations under micro-structure noise, we work only on the one-dimensional It\^o process case for simplicity.
We set 
\begin{align*}
    X_t = X_0 + \int_0^t \sigma (s) \, \mathrm{d}W_s 
    + \int_0^t \mu (s) \, \mathrm{d}s, \qquad t > 0, 
\end{align*}
where $ W $ is a one-dimensional Wiener process, 
$ \sigma $ and $ \mu $ are square- and absolute- integrable processes adapted to the filtration generated by $ W $, respectively.
Observations are set to be 
\begin{align*}
    Y_{k/n} = X_{k/n} + v_k, \qquad
    k=0,1, \dots, n.
\end{align*}

Even though the process $ Y $ is not a Gaussian process since $ \sigma $ is not deterministic in general, 
it turned out to be beneficial 
to consider the likelihood function of $ \mathbf{\Delta Y} := \{ \Delta Y_k =Y_{(k+1)/n} -Y_{k/n}, k=0, 1, \dots, n-1\} $ 
pretending that $ \sigma $ is a deterministic constant (we write $ \sigma^2 =:c $), $ \mu \equiv 0 $ and $ v_k \sim N (0, \nu) $, 
\underline{$k=1, 2, \dots, n $}, 
are zero-mean Gaussian random variables.

Under these specifications, we have 
\if0
\begin{align*}
    \begin{aligned}
        \mathbf{E}[\Delta Y_k\Delta Y_l] = \left(\frac{\sigma^2}{n} + 2 \nu \right)\delta_{k,l}  
        - \nu( \delta_{k+1,l} 1_{\{k\leq n-2\}} + \delta_{k,l+1} 1_{\{l\leq n-2\}}), 0 \leq k,l \leq n-1,  
    \end{aligned}
\end{align*}
where $ \delta_{\cdot, \cdot} $ is the Kronecker delta.
That is, 
\fi
\begin{align*}
    \mathbf{\Delta Y} \sim N \left(0, \left(\frac{c}{n}
+ 2 \nu \right) I_n - \nu J_n \right),
\end{align*}
where $ I_n $ is the $ n \times n $ identity matrix, and 
the $ n \times n $ matrix $ J_n $ is defined by
\begin{align}\label{Jacobi1}
    J_n := 
    \begin{pmatrix}
        1 & 1 & 0 & \cdots & 0 & 0 & 0 \\
        1 & 0 & 1 & \cdots & 0 & 0 & 0 \\
        0 & 1 & 0 & \ddots & 0 & 0 & 0 \\
        \vdots & \vdots & \ddots & \ddots & \ddots & \vdots & \vdots \\
        0 & 0 & 0 & \ddots & 0 & 1 & 0 \\
        0 & 0 & 0 & \cdots & 1 & 0 & 1 \\
        0 & 0 & 0 & \cdots & 0 & 1 & 0
    \end{pmatrix}. 
\end{align}
It is noteworthy that 
the matrix $ \mathbf{P}_n $ 
introduced in Section \ref{SIMLintro} does diagonalize $ J_n $.
Moreover, the diagonal is decreasing. To be precise, we have 
\begin{align*}
    \begin{aligned}
        J_n = \mathbf{P}_n 
        \mathrm{diag}
            \left[ 2\cos \left(\frac{\pi}
{2n+1}\right), 2\cos \left(\frac{3\pi}
{2n+1}\right), \dots,  2\cos \left(\frac{(2n-1)\pi}
{2n+1}\right)\right]
        \mathbf{P}_n^\top,
    \end{aligned}
\end{align*}
where we denote by $ \mathrm{diag}[b_1, b_2, \dots, b_n] $
the diagonal matrix whose components are $ b_1, b_2, \dots, b_n $, 
and $\mathbf{P}_n^\top$ means the transpose of $\mathbf{P}_n$.

We can then write down the log-likelihood function 
of $ \mathbf{z} = n^{-1/2}
\mathbf{P}_n \mathbf{\Delta Y}_n $ as 
\begin{align*}
    \begin{aligned}
        L_n (c, \nu) 
       := -\frac{1}{2} \sum_{k=1}^n \log (a_{k,n} \nu + c) -\frac{1}{2} \sum_{k=1}^n \frac{\mathbf{z}_k^2}{a_{k,n}\nu + c},    
    \end{aligned}
\end{align*}
where we put
\begin{align*}
    \begin{aligned}
        a_{k,n} := 4 n \sin^2 \left(\frac{2k-1}{2n+1}\frac{\pi}{2}  \right), \quad k=1, 2, \dots, n. 
    \end{aligned}
\end{align*}

Kunitomo and Sato \cite{KS08a}
introduced the following decomposition of $ L_n $:
\begin{align*}
    \begin{aligned}
        2 L_n = L_n^{(1)} + L_n^{(2)} + L_n^{(r)},
    \end{aligned}
\end{align*}
where we put
\begin{align}\label{KSL0}
    \begin{aligned}
        L_n^{(1)} (c) 
        := -m \log c - \frac{1}{c}\sum_{k=1}^n \mathbf{z}_k^2,  
    \end{aligned}
\end{align}
which we call the prototype of the KS likelihood function, 
\begin{align*}
    \begin{aligned}
        L_n^{(2)} (\nu) 
        := - \sum_{k=n+1-l}^n \log (a_{k,n} \nu ) -\frac{1}{\nu} \sum_{k=n+1-l}^n \frac{\mathbf{z}_k^2}{a_{k,n}},
    \end{aligned}
\end{align*}
and $ L_n^{(r)} := 2 L_n - L_n^{(1)} - L_n^{(2)}$.
Note that the decomposition depends on the choice of $ (m,l) $. 
The term {\it separating information} comes from this decomposition. 

Now we see that the SIML estimator \eqref{SIMLestimator} maximize $ L_n^{(1)} $, and 
\begin{align*}
    \nu^*_{n,l} := \frac{1}{l} \sum_{k=n+1-l}^n \frac{\mathbf{z}_k^2}{a_{k,n}}
\end{align*}
maximize $ L_n^{(2)} $. 
With a proper choice of $ (m_n, l_n) $, 
we may have 
\begin{align*}
    \lim_{n \to \infty} L_n^{(r)} (V_{n,m_n}, \nu^*_{n,l_n} ) =0,
\end{align*}
so that the estimators are asymptotically optimal
(see \cite[Chapter 3]{KSK} for a discussion).

%%%%%%%%%%%%%%%
\section{Malliavin--Mancino Method and its KS likelihood}\label{MM-KS}
\subsection{Malliavin--Mancino's Fourier method}\label{MM}
The Malliavin-Mancino's Fourier method (MM method, for short), 
introduced in \cite{MMFS,BMM, MMAS}
(see also \cite{MRS}, which is a good review of the literature), 
is an estimation method 
for the spot volatility $ \Sigma (s) $
in Section~\ref{Problem}, 
by constructing an estimator of the 
Fourier series of $ \Sigma $. 
The series consists of estimators 
of Fourier coefficients
given by 
\begin{equation}\label{Fourier-estimator}
    \begin{split}
&\widehat{\Sigma^{j,j'}_{n,m_n}}(q) \\ 
&:=\frac{1}{2m_n+1} \sum_{|l|\leq m_n} \left( \sum_{k=1}^{n^j}e^{2\pi\sqrt{-1} (l+q) t^j_{k-1} }
        \Delta Y^j_k \right)
        \left(\sum_{k'=1}^{n^{j'}}e^{-2\pi\sqrt{-1} l t_{k'-1}^{j'} }  \Delta Y^{j'}_{k'} \right) 
    \end{split}
\end{equation}
for $j, j'=1, 2, \dots, J$ and $ q \in \mathbf{Z} $. 
We note that 
$ \widehat{\Sigma^{j,j'}_{n,m_n}}(0)$
is a variant of 
the SIML estimator \eqref{SIMLestimator}. 
In fact, 
we have 
\begin{align*}
    \begin{aligned}
       & \widehat{\Sigma^{j,j'}_{n,m_n}}(0) \\
    &=\frac{1}{2m_n+1} \sum_{|l|\leq m_n} \left( \sum_{k=1}^{n^j}e^{2\pi\sqrt{-1} q t^j_{k-1} }
        \Delta Y^j_k \right)
    \left(\sum_{k'=1}^{n^{j'}}e^{-2\pi\sqrt{-1} l t_{k'-1}^{j'} }  \Delta Y^{j'}_{k'} \right) \\
    &= \frac{1}{2m_n+1} \Bigg\{
     \left( \sum_{k=1}^{n^j}
        \Delta Y^j_k \right) 
    \left(\sum_{k'=1}^{n^{j'}}\Delta Y^{j'}_{k'} \right) \\
   & \qquad  + 
    2\sum_{l=1}^{m_n}
    \left( \sum_{k=1}^{n^j}\cos {2\pi l t^j_{k-1} }
        \Delta Y^j_k \right) 
    \left(\sum_{k'=1}^{n^{j'}}\cos {2\pi l t_{k'-1}^{j'} }  \Delta Y^{j'}_{k'} \right)
    \\
    & \qquad  + 2 \sum_{l=1}^{m_n}\left( \sum_{k=1}^{n^j}\sin {2\pi l t^j_{k-1} }
        \Delta Y^j_k \right) 
    \left(\sum_{k'=1}^{n^{j'}}\sin {2\pi l t_{k'-1}^{j'} }  \Delta Y^{j'}_{k'} \right)
    \Bigg\},
    \end{aligned}
\end{align*}
and so, if we put
\begin{align*}
    \begin{aligned}
         q^j_{k,l} :=  
         \begin{cases}
             n_j^{-1/2} & l=0 \\
             n_j^{-1/2} \sqrt{2} \cos (2 \pi l t^j_{k-1}) & l=2,4,6,  \dots \\
             n_j^{-1/2} \sqrt{2} \sin (2 \pi l t^j_{k-1}) & l= 1, 3, 5, \dots, 
         \end{cases}
    \end{aligned}
\end{align*}
for $ k=1,2, \cdots, n_j $, we have 
\begin{align}\label{MMinteg}
    \begin{aligned}
        \widehat{\Sigma^{j,j'}_{n,m_n}}(0)
        =  \frac{ \sqrt{n_j n_{j'}} }{2 m_n +1} \sum_{l=0}^{2m_n} \sum_{k=1}^{n_j} q^{{j}}_{k,l} \Delta Y_k^j  \sum_{k'=1}^{n_{j'}} q^{{j'}}_{k,l} \Delta Y_{k'}^{j'}. 
    \end{aligned}
\end{align}

Now we see the similarity between the SIML estimator \eqref{SIMLestimator}
and the MM estimator \eqref{MMinteg}.
In the following section, we show that the difference of the two estimators 
can be explained in terms of the KS likelihood function. 

\subsection{A KS likelihood function of the MM method}
From now on, we consider the case where 
$ d=1 $, so that the dependence on the coordinate $ j $ vanishes, 
where $ n $ is odd, so that we rewrite it as $2n +1 $, 
and where $ t_k = k/(2n+1) $ for $ k=0,1, \dots, 2n+1 $.
Then, we obtain 
a $ (2n+1) \times (2n+1)$ matrix $ \mathbf{Q}_n:= (q^{2n+1}_{k,l})_{0\leq k,l \leq 2n}$
with 
\begin{align*}
      \begin{aligned}
         q^{2n+1}_{k,l} :=  
         \begin{cases}
           ( 2 n+1)^{-1/2} & l=0 \\
             ( 2 n+1)^{-1/2}  \sqrt{2} \cos \left(\dfrac{l k \pi}{2n+1}\right) & l=2, 4, \dots, 2n \vspace{2mm}\\
          ( 2 n+1)^{-1/2}  \sqrt{2} \sin \left(\dfrac{(l+1) k \pi}{2n+1}\right) & l=1, 3, \dots, 2n-1. 
         \end{cases}
    \end{aligned}
\end{align*}
\begin{lemma}
The matrix $ \mathbf{Q}_n $ is an orthogonal matrix and 
\begin{align*}
    \begin{aligned}
        &\mathbf{Q}_n \mathrm{diag} \Bigg[2, 2\cos \left(\frac{2\pi}{2n+1}\right), 2\cos \left(\frac{2\pi}{2n+1}\right), \\
        &\hspace{1cm}\dots, 2\cos \left(\frac{2n\pi}{2n+1}\right),2\cos \left(\frac{2n \pi}{2n+1}\right) \Bigg] \mathbf{Q}_n^\top = \widetilde{J}_n,
    \end{aligned}
\end{align*}
where the $(2n+1) \times (2n+1)$ matrix $\widetilde{J}_n$ is defined by
\begin{align}\label{JacobiMM}
    \widetilde{J}_n := 
    \begin{pmatrix}
        0 & 1 & 0 & \cdots & 0 & 0 & 1 \\
        1 & 0 & 1 & \cdots & 0 & 0 & 0 \\
        0 & 1 & 0 & \ddots & 0 & 0 & 0 \\
        \vdots & \vdots & \ddots & \ddots & \ddots & \vdots & \vdots \\
        0 & 0 & 0 & \ddots & 0 & 1 & 0 \\
        0 & 0 & 0 & \cdots & 1 & 0 & 1 \\
        1 & 0 & 0 & \cdots & 0 & 1 & 0
    \end{pmatrix}. 
\end{align}
\end{lemma}

\begin{proof}
It is a consequence of the following 
elementary equations:
\begin{align*}
    \begin{aligned}
        2 \cos \left(\frac{2 l \pi}{2n+1}\right)
           \cos \left(\frac{2 k l \pi}{2n+1}\right)
           = \cos \left(\frac{2 l(k+1) \pi}{2n+1}\right) + \cos \left(\frac{2 l(k-1) \pi}{2n+1}\right)
    \end{aligned}
\end{align*}
and 
\begin{align*}
    \begin{aligned}
        2 \cos \left(\frac{2 l \pi}{2n+1}\right)
           \sin \left(\frac{2 k l \pi}{2n+1}\right)
           = \sin \left(\frac{2 l(k+1) \pi}{2n+1}\right) + \sin \left(\frac{2 l(k-1) \pi}{2n+1}\right),
    \end{aligned}
\end{align*}
which are valid for any $ l $ and $ k $ modulo $ 2n+1 $. 
\end{proof}

The above lemma 
insists that the KS likelihood function 
associated with the MM estimator $ \widehat{\Sigma^{j,j'}_{n,m_n}}(0) $ of \eqref{MMinteg}
is the one obtained by replacing $ J_n $
of \eqref{Jacobi1}
with \eqref{JacobiMM}.

\subsection{Inconsistency caused by the initial and the last noises}
Looking back the procedure 
for constructing KS likelihood function
from the observation structure in Section~\ref{KSL}, 
we notice that 
the matrix $ \widetilde{J}_n $
might come from the equation 
$ v_n =v_0 $, which is of course unrealistic. 
The following proposition
may reflect this structure. 

\begin{proposition}\label{prop2-1}
We assume that $ v_0 $ and $ v_{2n+1} $ are independent 
and also independent of 
$ v_1, \cdots, v_{2n} $. 
Then, for any pair of $ (n, m_n) $ such that $ m_n < n $, we have 
\begin{align}\label{lb2-1}
    \begin{aligned}
         \mathbf{E} \left[  \frac{2n+1}{2m_n+1} \sum_{l=0}^{2m_n} \left( \sum_{k=1}^{2n+1}
        q^{2n+1}_{k-1,l} \Delta v_k \right)^2\right] \geq  \mathbf{E}[(v_0)^2 + (v_{2n+1})^2]. 
    \end{aligned}
\end{align}
\end{proposition}

\begin{proof}
As we did in the proof of 
Proposition \ref{prop1}, 
we observe 
\begin{align}\label{lb2-0}
    \begin{aligned}
       & \mathbf{E}\left[  \left( \sum_{k=1}^{2n+1}
        q^{2n+1}_{k-1,l} \Delta v_k \right)^2\right] \\
        & =\mathbf{E}\left[  \left( \sum_{k=1}^{2n}
        (q^{2n+1}_{k-1,l} -q^{2n+1}_{k,l}) v_k 
        + q^{2n+1}_{2n,l} v_{2n+1} - q^{2n+1}_{0,l} v_0 \right)^2\right] \\
    %   & = \mathbf{E}[v_1^2]       \left( \sum_{k=1}^{n-1}        (p^{n}_{k,l} -p^n_{k+1,l})^2         + (p^n_{n,l})^2 + (p^n_{1,l})^2 \right)
      & \geq \mathbf{E}[v_{2n+1}^2] (q^{2n+1}_{2n,l})^2
      + \mathbf{E}[v_0^2] (q^n_{0,l})^2. 
    \end{aligned}
\end{align}
Since it holds that
\begin{align*}
      \begin{aligned}
         q^{2n+1}_{0,l} :=  
         \begin{cases}
           ( 2 n+1)^{-1/2} & l=0 \\
             ( 2 n+1)^{-1/2}  \sqrt{2}  & l=2, 4, \dots, 2n \\
          0 & l=1, 3, \dots, 2n-1, 
         \end{cases}
    \end{aligned}
\end{align*}
and 
\begin{align*}
      \begin{aligned}
         q^{2n+1}_{2n,l} :=  
         \begin{cases}
           ( 2 n+1)^{-1/2} & l=0 \\
             ( 2 n+1)^{-1/2}  \sqrt{2} \cos \left(\dfrac{2 n l \pi}{2n+1}\right) & l=2, 4, \dots, 2n  \vspace{2mm}\\
          ( 2 n+1)^{-1/2}  \sqrt{2} \sin \left(\dfrac{2 n(l+1) \pi}{2n+1}\right) & l=1, 3, \dots, 2n-1, 
         \end{cases}
    \end{aligned}
\end{align*}
we have 
\begin{align*}
    \begin{aligned}
   %\frac{2n+1}{2m_n+1}
   \sum_{l=0}^{2m_n}(q^n_{0,l})^2
   =   \frac{2m_n+1}{2n+1}
    \end{aligned}
\end{align*}
and 
\begin{align*}
    \begin{aligned}
   %\frac{2n+1}{2m_n+1}
   \sum_{l=0}^{2m_n}(q^n_{2n,l})^2
   =   \frac{1}{2n+1}
   + \frac{2}{2n+1} 
   \sum_{l'=1}^{m_n}
   \left\{\cos^2 \left(\frac{2 \pi l' \pi}{2n+1}\right) + \sin^2 \left(\frac{2 \pi l' \pi}{2n+1}\right) \right\}
   =  \frac{2m_n+1}{2n+1}, 
    \end{aligned}
\end{align*}
which ensure \eqref{lb2-1}. 
\end{proof}

\begin{remark}
The robustness of the MM estimator against the micro-strucrue noise has been well recognised (see \cite{MS1, MS2}). 
The inequality \eqref{lb2-1}, 
together with the inequality \eqref{lb2-0},
states that 
the consistency of the MM estimator cannot be established without assuming 
$ v_0 =v_{2n+1} $, 
but this does not contradict the observations given in \cite{MS1, MS2}. 
\end{remark}
%%%%%%%%%%%%%%%%
\section{Initial-Noise Adjusted Estimator}\label{INA}

\subsection{Adjustment for the initial noise}
As we have seen, 
the KS estimator \eqref{SIMLestimator} cannot be consistent
if $ v_0 \ne 0 $. 
In this section, we propose an adjusted estimator
by replacing $ J_n $ of \eqref{Jacobi1}
with
\begin{align*}%\label{Jacobi2}
    \widetilde{J}_n' := 
    \begin{pmatrix}
        0 & 1 & 0 & \cdots & 0 & 0 & 0 \\
        1 & 0 & 1 & \cdots & 0 & 0 & 0 \\
        0 & 1 & 0 & \ddots & 0 & 0 & 0 \\
        \vdots & \vdots & \ddots & \ddots & \ddots & \vdots & \vdots \\
        0 & 0 & 0 & \ddots & 0 & 1 & 0 \\
        0 & 0 & 0 & \cdots & 1 & 0 & 1 \\
        0 & 0 & 0 & \cdots & 0 & 1 & 0
    \end{pmatrix},
\end{align*}
which reflects 
$ v_0 \sim N (0, \nu) $
in defining the KS likelihood function.

\begin{lemma}
We define an $ n \times n $ matrix 
    $ \mathbf{R}_n = (r^n_{l,k}) $ by
    \begin{align}\label{diag}
    \begin{aligned}
        r^n_{l,k} := 
       \sqrt{\frac{2}{n+1}} \sin \left( \frac{kl\pi}{n+1}\right), \qquad
        l, k=1, 2, \dots, n.
    \end{aligned}
    \end{align}   
Then, $ \mathbf{R}_n $ is an orthogonal matrix and 
\begin{align*}
    \begin{aligned}
         \widetilde{J}_n' = \mathbf{R}_n 
        \mathrm{diag}
            \left[2\cos \left(\frac{\pi}{n+1}\right), 2\cos \left(\frac{2\pi}{n+1}\right), 
            \dots,  2\cos \left(\frac{n\pi}{n+1}\right)\right]
        \mathbf{R}_n^\top.
    \end{aligned}
\end{align*}
\end{lemma}

\begin{proof}
    It is obvious from the equation 
    \begin{align*}
        \begin{aligned}
            2 \cos \left(\frac{k \pi }{n+1}\right)
            \sin \left(\frac{k l \pi }{n+1}\right)
            = \sin \left(\frac{k (l+1)\pi}{n+1}\right)
            + \sin \left(\frac{k (l-1)\pi}{n+1}\right),
        \end{aligned}
    \end{align*}
for $ k,l=1, 2, \dots, n $, and 
\begin{align*}
    \begin{aligned}
        \sum_{l=1}^n \sin^2 \left(\frac{k l \pi }{n+1}\right)
        &= \frac{1}{2 }\sum_{l=1}^n
        \left\{ 1 -\cos \left(\frac{ k l \pi }{n+1}\right) \right\} \\
        & = \frac{n}{2} - \frac{1}{4 }\sum_{l=1}^n
        \left(e^{2 \pi \frac{kl}{n+1}} + e^{-2 \pi \frac{kl}{n+1}}\right) \\
        & = \frac{n+1}{2}
    \end{aligned}
\end{align*}
for $ k=1, 2, \dots, n $. 
\end{proof}

\subsection{Initial-noise adjusted estimator}
Given \eqref{diag}, 
we may consider an estimator of 
the integrated volatility $
\int_0^1 \Sigma^{j,j'} (s) \,\mathrm{d}s $ gievn by
\begin{equation}\label{na-SIMLestimator}
    \begin{split}
        \tilde{V}_{n,m_n}^{j,j'}
        := \frac{n+1}{m_n} \sum_{l=1}^{m_n} \left( \sum_{k=1}^{n^j}
        r^{n^j}_{k,l} \Delta Y^j_k \right)
        \left(\sum_{k'=1}^{n^{j'}} r^{n^{j'}}_{k',l} \Delta Y^{j'}_{k'} \right)  
    \end{split}
\end{equation}
instead of \eqref{SIMLestimator}.
The good news is that, 
in contrast with Proposition \ref{prop1}, 
we have
\begin{proposition}
Assume that $ m_n = o (n^{1/2}) $ as $ n \to \infty $. 
Then, we have 
\begin{align}\label{lb2}
    \begin{aligned}
       \lim_{n \to \infty}  \mathbf{E} \left[  \frac{n+1}{m_n} \sum_{l=1}^{m_n} \left( \sum_{k=1}^{n}
        r^{n}_{k,l} \Delta v_k \right)^2\right] =0.  
    \end{aligned}
\end{align}
\end{proposition}

\begin{proof}
Let us start with observing  
\begin{align*}
    \begin{aligned}
       & \mathbf{E}\left[  \left( \sum_{k=1}^{n}
        r^{n}_{k,l} \Delta v_k \right)^2\right] \\
        & =\mathbf{E}\left[  \left( \sum_{k=1}^{n-1}
        (r^{n}_{k,l} -r^n_{k+1,l}) v_k 
        + r^n_{n,l} v_n - r^n_{1,l} v_0 \right)^2\right] \\
       & = \mathbf{E}[v_1^2]
       \left\{ \sum_{k=1}^{n-1}
        (r^{n}_{k,l} -r^n_{k+1,l})^2 
        + (r^n_{n,l})^2 + (r^n_{1,l})^2
       \right\} \\
       & = \frac{2\mathbf{E}[v_1^2]}{n+1} \sum_{k=1}^{n-1} 
      \left\{ \sin \left(\frac{kl\pi}{n+1}\right) -
      \sin \left(\frac{(k+1)l\pi}{n+1}\right)\right\}^2
      +  \frac{4\mathbf{E}[v_1^2]}{n+1} \sin^2 \left(\frac{l\pi}{n+1}\right).
    \end{aligned}
\end{align*}
Since it holds that
\begin{align*}
    \begin{aligned}
       & \sum_{k=1}^{n-1} 
      \left\{ \sin \left(\frac{kl\pi}{n+1}\right) -
      \sin \left(\frac{(k+1)l\pi}{n+1}\right)\right\}^2 \\
      & = 4 \sin^2 \left(\frac{l\pi}{2(n+1)}\right) \sum_{k=1}^{n-1} \cos^2 \left(\frac{(2k+1)l \pi}{2(n+1)}\right) 
      \leq \frac{l^2 \pi^2}{n+1}
   %   &=  2 \sin^2 \frac{l\pi}{2(n+1)} \sum_{k=1}^{n-1} \left( 1+ \cos \frac{(2k+1)l \pi}{n+1}       \right) \\
    %  &= 2 \sin^2 \frac{l\pi}{2(n+1)}  \left( n-1- 2 \cos \frac{l \pi}{n+1}      \right)
    \end{aligned}
\end{align*}
and 
\begin{align*}
    \sin^2 \left(\frac{l\pi}{n+1}\right) \leq \frac{l^2 \pi^2}{(n+1)^2}
\end{align*}
for $ l \leq m \ll n $, 
we obtain
\begin{align*}
    \begin{aligned}
       & \mathbf{E} \left[  \frac{n+1}{m_n} \sum_{l=1}^{m_n} \left( \sum_{k=1}^{n}
        r^{n}_{k,l} \Delta v_k \right)^2\right] \\
        &\leq \frac{2\mathbf{E}[v_1^2] \pi^2}{m_n} 
        \left\{ \frac{1}{(n+1)} + \frac{1}{(n+1)^2}
        \right\}\sum_{l=1}^m l^2, 
    \end{aligned}
\end{align*}
which proves \eqref{lb2}. 
\end{proof}

\section{Concluding Remark}\label{conc}
We have accomplished the aim of the paper to introduce and show the benefits of the approach based on the KS likelihood function, which should be generalized to a more sophisticated entropy argument in the future.

\end{document}